\documentclass[11pt]{article}
\usepackage[utf8]{inputenc}
\usepackage{amsmath,amssymb,amsthm,mathabx}
\usepackage{fullpage}
\usepackage{hyperref}
\usepackage{authblk}

\DeclareMathOperator*{\EE}{\mathbb{E}}
\DeclareMathOperator*{\VV}{\mathbb{V}}
\DeclareMathOperator{\dist}{dist}

\providecommand{\RR}{\mathbb{R}}
\providecommand{\fun}{\ell}
\providecommand{\gun}{r}

\providecommand{\cS}{\mathcal{S}}

\newtheorem{theorem}{Theorem}[section]
\newtheorem{lemma}[theorem]{Lemma}
\newtheorem{corollary}[theorem]{Corollary}

\theoremstyle{definition}
\newtheorem{definition}[theorem]{Definition}

\title{Friedgut--Kalai--Naor theorem for slices of the Boolean cube}
\author{Yuval Filmus\thanks{Research conducted at the Simons Institute for the Theory of Computing during the 2013 fall semester on Real Analysis in Computer Science, and at the Institute for Advanced Study, Princeton, NJ. This material is based upon work supported by the National Science Foundation under agreement No.~DMS-1128155. Any opinions, findings and conclusions or recommendations expressed in this material are those of the authors, and do not necessarily reflect the views of the National Science Foundation.}}
\affil{Technion --- Israel Institute of Technology, Haifa, Israel}

\begin{document}

\maketitle

\begin{abstract}
The Friedgut--Kalai--Naor theorem, a basic result in the field of analysis of Boolean functions, states that if a Boolean function on the Boolean cube $\{0,1\}^n$ is close to a function of the form $c_0 + \sum_i c_i x_i$, then it is close to a dictatorship (a function depending on a single coordinate). We prove an analogous theorem for functions defined on the slice $\binom{[n]}{k} = \{ (x_1,\ldots,x_n) \in \{0,1\}^n : \sum_i x_i = k \}$.

When $k/n$ is bounded away from~$0$ and~$1$, our theorem states that if a function on the slice is close to a function of the form $\sum_i c_i x_i$ then it is close to a dictatorship. When $k/n$ is close to~$0$ or to~$1$, we can only guarantee being close to a junta (a function depending on a small number of coordinates); this deterioration in the guarantee is unavoidable, since for small $p$ a maximum of a small number of variables is close to their sum.

Kindler and Safra proved an FKN theorem for the biased Boolean cube, in which the underlying measure is the product measure $\mu_p(x) = p^{\sum_i x_i} (1-p)^{\sum_i (1-x_i)}$. As a corollary of our FKN theorem for the slice, we deduce a uniform version of the FKN theorem for the biased Boolean cube, in which the error bounds depend uniformly on $p$. Mirroring the situation on the slice, when $p$ is very close to~$0$ or to~$1$, we can only guarantee closeness to a junta.
\end{abstract}

\section{Introduction} \label{sec:intro}

Analysis of Boolean functions is a research area at the intersection of combinatorics, probability theory, functional analysis, and theoretical computer science. It traditionally studies real-valued functions on the Boolean cube $\{0,1\}^n$; often these functions are Boolean, that is, $\{0,1\}$-valued.

Recently interest has arisen in generalizing classical results in the area from functions on the Boolean cube to functions on other domains. While the theory for product domains such as $\{1,\ldots,k\}^n$ is similar to the theory for the Boolean cube, non-product domains such as symmetric groups and slices (defined below) present novel difficulties, the most conspicuous of which being the absence of a canonical Fourier basis.

In this paper, our object of study is \emph{slices of the Boolean cube}:
\[
 \binom{[n]}{k} = \bigl\{(x_1,\ldots,x_n) : \sum_{i=1}^n x_i = k \bigr\}.
\]
(Here $[n]$ stands for the set $\{1,\ldots,n\}$.)
Slices arise naturally in extremal combinatorics (the Erd\H{o}s--Ko--Rado theorem), graph theory ($G(n,M)$ graphs), and coding theory (constant-weight codes). They are also one of the simplest association schemes.

Recently the study of analysis of Boolean functions on the slice has gained traction, as witnessed by several recent articles~\cite{ODonnellWimmer,Wimmer,FilmusEJC,FKMW,FM,Srinivasan}. Classical theorems in the area which have been generalized to the slice include, among else, the Kahn--Kalai--Linial theorem~\cite{KKL,ODonnellWimmer}, Friedgut's junta theorem~\cite{FriedgutJunta,Wimmer,FilmusEJC}, and the Mossel--O'Donnell--Oleszkiewicz invariance principle~\cite{MOO,FKMW,FM}; see O'Donnell~\cite{ODonnell} for a description of these results (in their classical form).

This paper continues the project of generalizing analysis of Boolean functions to functions on the slice by proving a slice analog of the fundamental structural result of Friedgut, Kalai and A.~Naor~\cite{FKN}, which states that if a Boolean function on the Boolean cube $\{0,1\}^n$ is $\epsilon$-close to an affine function (a function of the form $(x_1,\ldots,x_n) \mapsto c_0 + \sum_{i=1}^n c_i x_i$) then it is $O(\epsilon)$-close to one of the functions $0,1,x_i,1-x_i$. (Two functions $f,g$ are $\epsilon$-close if $\|f-g\|^2 \leq \epsilon$, where $\|\cdot\|$ is the $L_2$ norm.)

Our main theorem states that for $2 \leq k \leq n-2$ and $p = \min(k/n,1-k/n)$, if a function $f\colon \binom{[n]}{k} \to \{0,1\}$ is $\epsilon$-close to an affine function for $\epsilon = O(p^2)$, then $f$ is $O(\epsilon)$-close to one of the functions $0,1, x_i,1-x_i$. For larger $\epsilon$, we prove that either $f$ or $1-f$ is $O(\epsilon)$-close to $\max_{i \in S} x_i$ (when $k \leq n/2$) or to $\min_{i \in S} x_i$ (when $k \geq n/2$) for some set $S$ of size $|S| = O(\sqrt{\epsilon}/p)$.

(The logically inclined reader can read $\bigwedge_{i \in S} x_i$ for $\max_{i \in S} x_i$ and $\bigvee_{i \in S} x_i$ for $\min_{i \in S} x_i$.)

%The FKN theorem has been extended in many directions. Kindler and Safra~\cite{KindlerSafra} proved a similar theorem for almost low-degree functions, the original theorem corresponding to the almost linear case (see Kindler's PhD thesis~\cite{Kindler} for an alternative account). Kindler and Safra~\cite{KindlerSafra}, Jendrej, Oleszkiewicz and Wojtaszczyk~\cite{JOW}, Nayar~\cite{Nayar} and Rubinstein and Safra~\cite{RS} have extended the FKN theorem to Boolean functions which are close to a sum of general independent random variables. Ellis, Filmus and Friedgut~\cite{EFF1,EFF2} extended the FKN theorem to functions on the symmetric group.

\paragraph{Comparison to other FKN theorems}

The original FKN theorem~\cite{FKN} states that if a function $f\colon \{0,1\}^n \to \{0,1\}$ is close to an affine function, then it is close to a dictatorship (a function depending on at most one coordinate). Kindler and Safra~\cite{KindlerSafra,Kindler} extended this theorem to the $\mu_p$-setting (the so-called \emph{biased Boolean cube}), in which the distribution on $\{0,1\}^n$ is not uniform but is the product measure $\mu_p$, in which $\Pr[x_i = 1] = p$ for $i \in [n]$.
Their theorem states that for each $p$, if $f\colon \{0,1\}^n \to \RR$ is $\epsilon$-close to an affine function (with respect to the $\mu_p$ measure), then $f$ is $O(\epsilon)$-close to a dictatorship; the hidden constant depends on $p$.
(For more results in this vein, see~\cite{JOW,Nayar,Rubinstein,RS}.)

In contrast, our main theorem states that if $p \leq 1/2$ and $f\colon \binom{[n]}{pn} \to \{0,1\}$ is $\epsilon$-close to an affine function then it is $C \epsilon$-close to a dictatorship \emph{assuming $\epsilon = O(p^2)$}; for larger $\epsilon$, we only guarantee that $f$ depends on $O(\sqrt{\epsilon}/p)$ inputs.
The restriction to $\epsilon = O(p^2)$ is necessary, since the function $\max(x_1,x_2)$ is $p^2$-close to the affine function $x_1 + x_2$.
However, for \emph{fixed} $p$ we can obtain a statement similar to that of the original FKN theorem by letting the constant~$C$ depend on~$p$.

As stated above, the error bound in the FKN theorem of Kindler and Safra depends on $p$. In contrast, the error bound in our theorem is uniform over all $p$, and this is why we cannot guarantee closeness to a dictatorship. Using our theorem as a black box, we prove a uniform version of the FKN theorem of Kindler and Safra, whose form is identical to that of our theorem for the slice.

%The advantage of our formulation over the usual formulation is that our FKN theorem provides a uniform bound for every $p$. In contrast, the usual formulation of the FKN theorem only works for specific (or at least bounded) $p$. %We note, however, that the proof of the FKN theorem in Kindler and Safra~\cite{KindlerSafra} also provides such a uniform bound, although this bound is not explicitly stated.

%Another advantage of our formulation is that it is meaningful for $p = o(\sqrt{\epsilon})$, in which case $f$ can only be guaranteed to be close to a junta (a function depending on a small number of inputs) of the form $\max_{i \in S} x_i$ or $1-\max_{i \in S} x_i$. One can ask for a similar formulation of the FKN theorem over the Boolean cube $\{0,1\}^n$ with respect to the $\mu_p$ measure, and indeed such a statement follows from our main theorem.

A similar situation occurs in the FKN theorem for permutations~\cite{EFF1,EFF2}. We can think of the set $S_n$ of permutations on $n$ points as a certain subset of $\binom{[n]^2}{[n]}$ consisting of those sets whose projections to the individual coordinates are equal to $[n]$. The density parameter $p$ thus has the value $p = n/n^2 = 1/n$, and as a result, given a function $f\colon S_n \to \{0,1\}$ which is close to an affine function, we can only guarantee that it is close to a function of the form $\max_{(i,j) \in S} x_{ij}$. When $f$ is balanced ($\EE[f] \approx 1/2$), however, we are able to guarantee that $f$ is close to a ``permutation-dictatorship''; we refer the interested reader to~\cite{EFF2}.

\paragraph{On the proof}

Our proof uses a novel proof method, the \emph{random sub-cube} method, which allows us to reduce the FKN theorem for the slice to the FKN theorem for the Boolean cube (see Keller~\cite{Keller} for a similar reduction from the biased $\mu_p$ measure on the Boolean cube to the uniform measure on the Boolean cube). The idea is to consider subsets of the slice which are isomorphic to a Boolean cube of dimension $k$ (assuming $k \leq n/2$):
\[
 \{a_1, b_1\} \times \cdots \times \{a_k,b_k\}.
\]
We can apply the classical FKN theorem on each of these sub-cubes. Moreover, if we choose $a_1,b_1,\ldots,a_k,b_k$ at random, then a uniform point on a uniform sub-cube is just a uniform point on the slice, and this allows us to deduce our FKN theorem.

\paragraph{Applications}

Our main theorem has recently been used by Das and Tran~\cite{DasTran} to determine the sharp threshold for the Erd\H{o}s--Ko--Rado property on a random hypergraph, improving on an earlier result of Bollob\'as et al.~\cite{BNR} which used the classical FKN theorem. For further work on the problem, see Devlin and Kahn~\cite{DK}.

%It is natural to ask which other classical theorems in analysis of Boolean functions extend to the slice (for on overview of the area, consult O'Donnell's recent monograph~\cite{ODonnell}). Prior to this article, O'Donnell and Wimmer~\cite{ODonnellWimmer} have generalized the KKL theorem~\cite{KKL}, and Wimmer~\cite{Wimmer} has generalized Friedgut's junta theorem~\cite{FriedgutJunta}. Following this article, Filmus~\cite{FilmusEJC} has described an explicit Fourier-like basis for the slice and used it to streamline Wimmer's junta theorem~\cite{Wimmer}, and Filmus et al.~\cite{FKMW,FM} have proved a version of the non-linear invariance principle~\cite{MOO} for the slice, and used it to derive other results such as the Majority is Stablest theorem~\cite{MOO} and the Kindler--Safra theorem~\cite{KindlerSafra}.

\paragraph{Paper organization} After some preliminary definitions appearing in Section~\ref{sec:prel}, we formally state our main theorem in Section~\ref{sec:statement}, where we also derive the uniform FKN theorem for the biased Boolean cube. The proof itself appears in Section~\ref{sec:proof}.

\paragraph{Acknowledgements} The author thanks Guy Kindler, Elchanan Mossel and Karl Wimmer for helpful discussions, Manh Tuan Tran for pointing out a mistake in an earlier version, and the anonymous reviewers for helpful suggestions.

\section{Preliminaries} \label{sec:prel}

\paragraph{Notations}
We use the notations $[n] = \{1,\ldots,n\}$ and $\dist(x,S) = \min_{y \in S} |x-y|$.

A \emph{Boolean function} is a $\{0,1\}$-valued function.
An \emph{affine function} is a function of the form
\[ \ell(x_1,\ldots,x_n) = c_0 + \sum_{i=1}^n c_i x_i. \]
Our main theorems involve the maximum and the minimum of a set of Boolean variables. \emph{We adopt the convention that $\max \emptyset = 0$ and $\min \emptyset = 1$.}

For a function $f$ on a finite domain $D$, the squared $L_2$ norm of $f$ is given by $\|f\|^2 = \EE[f^2]$, where the expectation is with respect to the uniform distribution over $D$. The squares norm satisfies a triangle inequality of the form $\|f+g\|^2 \leq 2(\|f\|^2 + \|g\|^2)$.

We say that two functions $f,g$ on the same domain are \emph{$\epsilon$-close} if $\|f-g\|^2 \leq \epsilon$. If $f,g$ are Boolean then they are $\epsilon$-close iff $\Pr[f \neq g] \leq \epsilon$. If $f$ is Boolean, $g$ is real-value, and $G$ is the Boolean function closest to $g$, then $\|f - G\|^2 = O(\|f - g\|^2)$. Note that $G$ is obtained by ``rounding'' $G$ to $\{0,1\}$, that is $G(x) = \operatorname{argmin}_{b \in \{0,1\}} |x-b|$.

\paragraph{Functions on the Boolean cube}
The \emph{Boolean cube} is the set $\{0,1\}^n$. We identify functions on the Boolean cube with functions on the Boolean-valued variables $x_1,\ldots,x_n$.

Each function on the Boolean cube has a unique Fourier expansion
\[
 f(x_1,\ldots,x_n) = \sum_{S \subseteq [n]} \hat{f}(S) \chi_S, \qquad
 \text{where } \chi_S = (-1)^{\sum_{i \in S} x_i}.
\]
Parseval's identity states that $\|f\|^2 = \sum_S \hat{f}(S)^2$.

\paragraph{Functions on the slice}
For integers $n \geq 2$ and $0 \leq k \leq n$, the \emph{slice} $\binom{[n]}{k}$ is defined as
\[
 \binom{[n]}{k} = \bigl\{ (x_1,\ldots,x_n) \in \{0,1\}^n : \sum_{i=1}^n x_i = k. \bigr\}
\]
Alternatively, we can think of $\binom{[n]}{k}$ as the collection of all subsets of $[n]$ of size exactly $k$, using the correspondence $(x_1,\ldots,x_n) \mapsto \{ i \in [n] : x_i = 1 \}$. We use this correspondence freely in the paper.

Every affine function on the slice has a unique representation of the form
\[ \ell(x_1,\ldots,x_n) = \sum_{i=1}^n c_i x_i. \]

\paragraph{The FKN theorem}

The Friedgut--Kalai--Naor theorem (FKN theorem for short) is the following result.

\begin{theorem}[Friedgut--Kalai--Naor~\cite{FKN}] \label{thm:fkn}
 Suppose $f\colon \{0,1\}^n \to \{0,1\}$ is $\epsilon$-close to an affine function. Then $f$ is $O(\epsilon)$-close to one of the functions $\{0,1,x_1,1-x_1,\ldots,x_n,1-x_n\}$.
\end{theorem}

\section{Statement of main theorem} \label{sec:statement}

Our main theorem is a version of Theorem~\ref{thm:fkn} for functions on a slice.

\begin{theorem} \label{thm:fkn-slice}
 Suppose $f\colon \binom{[n]}{k} \to \{0,1\}$ is $\epsilon$-close to an affine function, where $2 \leq k \leq n-2$. Define $p \triangleq \min(k/n,1-k/n)$. Then either $f$ or $1-f$ is $O(\epsilon)$-close to $\max_{i \in S} x_i$ (when $p \leq 1/2$) or to $\min_{i \in S} x_i$ (when $p \geq 1/2$) for some set $S$ of size at most $\max(1,O(\sqrt{\epsilon}/p))$.
\end{theorem}

(We remind the reader that by convention, if $S = \emptyset$ then $\max_{i \in S} x_i = 0$ and $\min_{i \in S} x_i = 1$.)

The statement implies that for some constant $C$, if $\epsilon < Cp^2$ then we are guaranteed that $|S| \leq 1$, and so $f$ can be approximated by a function of one of the forms $0,1,x_i,1-x_i$.

The bound on the size of $S$ is optimal: for $ps^2 = o(1)$, the function $\max(x_1,\ldots,x_s)$ is $O((ps)^2)$-close to the linear function $x_1+\cdots+x_s$, but cannot be approximated using a smaller maximum without incurring an error of $\Omega(p) = \omega((ps)^2)$.

When $k \in \{0,n\}$, the theorem is trivially true since every function is constant. When $k \in \{1,n-1\}$, the theorem is trivially true without the bound on $|S|$ since for $k = 1$ \emph{every} function on $\binom{[n]}{1}$ is affine, and every Boolean function $f$ on $\binom{[n]}{1}$ satisfies $f = \max \{x_i : f(\{i\}) = 1\}$; the case $k = n-1$ is similar.

\subsection{Uniform FKN theorem for the Boolean cube}

Theorem~\ref{thm:fkn-slice} can be used to derive a uniform biased version of Theorem~\ref{thm:fkn}.

\begin{definition} \label{def:mu-p}
 For each $n$, the measure $\mu_p$ is a measure on $\{0,1\}^n$ whose value on the atom $(x_1,\ldots,x_n)$ is $p^{\sum_i x_i} (1-p)^{\sum_i (1-x_i)}$.
\end{definition}

\begin{theorem} \label{thm:fkn-cube}
 Suppose $f\colon \{0,1\}^n \to \{0,1\}$ is $\epsilon$-close to an affine function with respect to the $\mu_p$ measure, for some $p \in (0,1)$. Then with respect to the $\mu_p$ measure, either $f$ or $1-f$ is $O(\epsilon)$-close to $\max_{i \in S} x_i$ (when $p \leq 1/2$) or to $\min_{i \in S} x_i$ (when $p > 1/2$) for some set $S$ of size at most $\max(1,O(\sqrt{\epsilon}/\min(p,1-p)))$.
\end{theorem}
\begin{proof}
 Suppose that $f$ is $\epsilon$-close to the affine function $\fun$. Let $N$ be a large integer (we will take the limit $N\to\infty$ later on), let $k_N = \lfloor pN \rfloor$, and consider the slice $\cS_N = \binom{[N]}{k_N}$. It is not hard to check that if $A$ is chosen randomly from $\cS$, then the distribution of $A \cap [n]$ tends (as a function of $N$) to the distribution $\mu_p$. Moreover, if we take $p_N = \min(k_N/N,1-k_N/N)$ then $p_N \to \min(p,1-p)$.
  
 Extend $f$ to a function $f_N$ on the slice $\cS_N$ by taking $f_N(x_1,\ldots,x_N) = f(x_1,\ldots,x_n)$, and extend $\fun$ to a function $\fun_N$ in a similar way. The remarks above show that $\|f_N - \fun_N\|^2 \leq 2\epsilon$ for large enough $N$. Also, for large enough $N$, $\min(k_N,N-k_N) = p_N N \geq 2$ (since $p_N \to p$). Therefore we can apply Theorem~\ref{thm:fkn-slice} to deduce that for large enough $N$, either $f_N$ or $1-f_N$ is $O(\epsilon)$-close to a maximum or a minimum of up to $\max(1,O(\sqrt{\epsilon}/p_N)) = \max(1,O(\sqrt{\epsilon}/p))$ inputs.
 
 Let $g_N$ denote the approximating function (the maximum or minimum of a small number of coordinates). If $g_N$ depends only on the first $n$ coordinates, and $g$ is the corresponding function on $\{0,1\}^n$, then $f$ or $1-f$ is $O(\epsilon)$-close to $g$, completing the proof. When $g_N$ depends on coordinates beyond the first $n$, there exists a substitution $\sigma$ to these coordinates such that $f_N|_\sigma$ or $1-f_N|_\sigma$ is $O(\epsilon)$-close to $g_N|_\sigma$. Since the number of substituted coordinates doesn't depend on $N$, we can complete the proof as before, noting that $g_N|_\sigma$ is either a maximum, a minimum, or a constant.
\end{proof}

The proof of Theorem~\ref{thm:fkn-cube} is similar to an argument of Ahlswede and Khachatrian~\cite{AK5}, in which the authors derive an Erd\H{o}s--Ko--Rado theorem on $\{1,\ldots,\alpha\}^n$ from a similar theorem for $\binom{[N]}{\alpha^{-1}N}$. Another version of the same argument is due to Dinur and Safra~\cite{DinurSafra}, who derive an Erd\H{o}s--Ko--Rado theorem on the Boolean cube $\{0,1\}^n$ with respect to $\mu_p$ from a similar theorem for $\binom{[N]}{pN}$.

Arguments going in the other direction are also known (for example, Friedgut~\cite{Friedgut}), but are more complicated and sometimes result in degredation of parameters. This highlights the fact that Theorem~\ref{thm:fkn-slice} is more general than Theorem~\ref{thm:fkn-cube}, in the sense that the former can be used to derive the latter, but not vice versa.

On the other hand, while we derived Theorem~\ref{thm:fkn-cube} from Theorem~\ref{thm:fkn-slice}, it is probably easier to prove it directly.
In particular, the estimates on hypergeometric distributions which are necessary for the proof of Theorem~\ref{thm:fkn-slice} (for example, Lemma~\ref{lem:hypergeometric-0} below) would be replaced with similar but simpler estimates on geometric distributions.

\subsection{The case \texorpdfstring{$\epsilon = 0$}{epsilon = 0}}

As a warm-up, we prove that the only Boolean affine functions on the slice are $0,1,x_i,1-x_i$.

\begin{lemma} \label{lem:cf-hard}
 Suppose $f\colon \binom{[n]}{k} \to \{0,1\}$ is affine, where $2 \leq k \leq n-2$. Then $f \in \{0,1\}$ or $f \in \{x_i,1-x_i\}$ for some $i$.
\end{lemma}
\begin{proof}
 Let $f = \sum_{i=1}^n c_i x_i$. Without loss of generality, suppose that $c_1 = \min(c_1,\ldots,c_n)$. For any $i \neq 1$, let $S \in \binom{[n]}{k}$ be some set containing $1$ but not $i$. Since $f(S \triangle \{1,i\}) - f(S) = c_i-c_1$ and $f$ is Boolean, we conclude that $c_i \in \{c_1,c_1+1\}$. If for all $i \neq 1$ we have $c_i = c_1$, then $f \in \{0,1\}$. So we can assume that $I_0 = \{ i \in [n] : c_i = c_1 \}$ and $I_1 = \{ i \in [n] : c_i = c_1 + 1 \}$ are both non-empty.

 We claim that either $|I_0| = 1$ or $|I_1| = 1$. Otherwise, suppose without loss of generality that $1,2 \in I_0$ and $3,4 \in I_1$ (here we are using the fact that $2 \leq k \leq n-2$). Let $S \in \binom{[n]}{k}$ be some set containing both $1,2$ but neither $3,4$. Then $f(S \triangle \{1,2,3,4\})-f(S) = c_3+c_4-c_1-c_2 = 2$, contradicting the fact that $f$ is Boolean. This shows that either $|I_0| = 1$ or $|I_1| = 1$. If $I_0 = \{1\}$ then
\[ f = c_1x_1 + \sum_{i=2}^n (c_1 + 1)x_i = (c_1 + 1)\sum_{i=1}^n x_i - x_1 = (c_1+1)k - x_1, \]
 and since $f$ is Boolean, necessarily $f = 1-x_1$. Similarly, if $I_1 = \{i\}$ then we get $f = x_i$.
\end{proof}

\section{Proof of main theorem} \label{sec:proof}

\subsection{Proof overview}

Since every function on $\binom{[n]}{k}$ is equivalent to a function on $\binom{[n]}{n-k}$, and the equivalence preserves affine functions, it suffices to consider the case $k \leq n/2$.

For the rest of this section, we make the assumption that $2 \leq k \leq n/2$ and fix the following notation:
\begin{itemize}
 \item $p = k/n \leq 1/2$.
 \item $f\colon \binom{[n]}{k} \to \{0,1\}$ is a Boolean function.
 \item $\fun = \sum_{i=1}^n c_i x_i$ is an affine function satisfying $\|f-\fun\|^2 \leq \epsilon$. % closest?
\end{itemize}

We first explain the proof of Theorem~\ref{thm:fkn-slice} in the easier case $\epsilon < p/128$. Extending Lemma~\ref{lem:cf-hard}, we show that the coefficients $c_1,\ldots,c_n$ are close to two values $\alpha,\alpha+1$, say most of them close to $\alpha$. We define $d_i = c_i$ or $d_i = c_i - 1$ in such a way that $d_1,\ldots,d_n$ are all close to $\alpha$, and let $\gun = \sum_i d_i x_i$. Note that $h = \fun-\gun$ is of the form $\sum_{i \in S} x_i$. Applying the classical Friedgut--Kalai--Naor theorem (in the form of Lemma~\ref{lem:fkn-consequence}) to a random sub-cube of $\binom{[n]}{k}$ (a subset of $\binom{[n]}{k}$ of the form $\bigtimes_{i=1}^k \{a_i,b_i\}$, where $a_1,b_1,\ldots,a_k,b_k$ are distinct elements of $[n]$), we deduce that
\[
 k \EE_{i \neq j} (d_i-d_j)^2 = k \EE_{i \neq j} \dist(c_i-c_j,\{0,\pm 1\})^2 = O(\epsilon).
\]
A simple calculation shows that $\VV \gun \leq k \EE_{i \neq j} (d_i-d_j)^2 = O(\epsilon)$, and so, putting $m = \EE \gun$, we get that
\[ \|f-(h+m)\|^2 \leq 2\|f-\fun\|^2 + 2\|\gun-m\|^2 = O(\epsilon). \]
This means that $f$ is close to the function $H$ obtained from rounding $h+m$ to $\{0,1\}$. The proof is complete by showing that the only way a function of the form $h+m$ is close to a Boolean function is when $H = \max_{i \in S} x_i$ (this includes the case $H = 0$) or $H = 1$.

When $\epsilon \geq p/128$ we cannot deduce that $d_i-d_j = \dist(c_i-c_j,\{0, \pm 1\})^2$. Instead, we start by showing that all but an $O(\epsilon)$ fraction of the coefficients $c_1,\ldots,c_n$ are concentrated around \emph{three} values $\alpha-1,\alpha,\alpha+1$. This allows us to approximate $f$ by a function of the form $\sum_{i \in S_+} x_i - \sum_{i \in S_-} x_i + m$. Further arguments show that one of the first two summands can be bounded by $O(\epsilon)$ in expectation, and the proof is completed as before.

\subsection{First steps}

\subsubsection{Concentration of coefficients}

We start by showing that for small $\epsilon/p$, the coefficients $c_1,\ldots,c_n$ are all concentrated around two values $\alpha,\alpha+1$.

\begin{lemma} \label{lem:cf-soft}
 There exist $c,d$ satisfying $|c-d| = 1$ and a subset $S \subseteq [n]$ of size $|S| \leq n/2$ such that for all $j \in S$, $|c_j - c|^2 \leq 8\epsilon/p$, and for all $j \notin S$, $|c_j - d|^2 \leq 8\epsilon/p$.
\end{lemma}
\begin{proof}
 Suppose, without loss of generality, that $c_1 = \min(c_1,\ldots,c_n)$, and fix $i \neq 1$. For any $T \in \binom{[n]}{k}$, define
\begin{align*}
 \Delta &= \dist(\fun(T),\{0,1\})^2 + \dist(\fun(T \triangle \{1,i\}),\{0,1\})^2 \\ &=
 \dist(\fun(T),\{0,1\})^2 + \dist(\fun(T) + c_i - c_1,\{0,1\})^2.
\end{align*}
 Let $r = \fun(T)$ and $\delta = c_i - c_1 \geq 0$. Suppose that $\fun(T)$ is closer to $a \in \{0,1\}$ and $\fun(T \triangle \{1,i\})$ is closer to $b \in \{0,1\}$, where $a \leq b$. Then
\[
 \Delta = (r-a)^2 + (r+\delta-b)^2 \geq \frac{(\delta-b+a)^2}{2},
\]
using the inequality $\alpha^2 + \beta^2 \geq (\alpha+\beta)^2/2$ with $\alpha = a-r$ and $\beta = r+\delta-b$.
Since $b-a \in \{0,1\}$, we deduce that
\[ \dist(\fun(T),\{0,1\})^2 + \dist(\fun(T \triangle \{1,i\}),\{0,1\})^2 \geq \frac{1}{2}\dist(c_i-c_1,\{0,1\})^2. \]
 Taking expectation over random $T$, we conclude that
\[
 \frac{1}{2}\dist(c_i-c_1,\{0,1\})^2 \Pr[1 \in T, i \notin T] \leq 2\epsilon.
\]
Since a random $T \in \binom{[n]}{k}$ contains $1$ but not $i$ with probability $\frac{k(n-k)}{n(n-1)} \geq p(1-p) \geq p/2$, we conclude that $\dist(c_i-c_1,\{0,1\})^2 \leq 8\epsilon/p$.

 For $\alpha \in \{0,1\}$, let $I_\alpha$ be the set of indices $i$ such that $c_i - c_1$ is closer to $\alpha$. The lemma now follows by either taking $S = I_0$, $c = c_1$ and $d = c_1+1$ or $S = I_1$, $c = c_1+1$ and $d = c_1$.
\end{proof}

\subsubsection{The random sub-cube argument}

Unconditionally, the values $c_i$ are on average either close to one another or at distance roughly $1$. We show this by taking a random sub-cube of dimension $k$, and applying the classical Friedgut--Kalai--Naor theorem to the restrictions of $f$ and $\fun$ to the sub-cube.

First, we need the following consequence of the FKN theorem.

\begin{lemma} \label{lem:fkn-consequence}
 Suppose $f\colon \{0,1\}^n \to \{0,1\}$ is $\epsilon$-close to an affine function $\ell\colon \{0,1\}^n \to \mathbb{R}$. Then
\[
 \sum_{i=1}^n \dist(2\hat{\ell}(\{i\}), \{0,\pm 1\})^2 = O(\epsilon).
\]
\end{lemma}
\begin{proof}
Theorem~\ref{thm:fkn} shows that $f$ is $O(\epsilon)$-close to some function $g$ of one of the forms $0,1,x_i,1-x_i$. The Fourier expansions of these functions are, respectively,
\[
 0, 1, \frac{1}{2} - \frac{1}{2} (-1)^{x_i}, \frac{1}{2} + \frac{1}{2} (-1)^{x_i}.
\]
In particular, $\hat{g}(\{i\}) \in \{0, \pm 1/2\}$ for all $i \in [n]$.

The triangle inequality shows that $\|g-\ell\|^2 \leq 2\|g-f\|^2 + 2\|f-\ell\|^2 = O(\epsilon)$. On the other hand, Parseval's identity shows that
\[ \|2(g-\ell)\|^2 \geq \sum_{i=1}^n [2\hat{g}(\{i\}) - 2\hat{\ell}(\{i\})]^2 \geq \sum_{i=1}^n \dist(2\hat{\ell}(\{i\}), \{0,\pm 1\})^2. \qedhere \]
%The corollary follows by combining both bounds.
\end{proof}

We can now apply the random sub-cube argument.

\begin{lemma} \label{lem:random-sub-cube}
 We have
\[
 k \EE_{i \neq j} \dist(c_i-c_j,\{0,\pm 1\})^2 = O(\epsilon).
\]
\end{lemma}
\begin{proof}
 Let $a_1,b_1,\ldots,a_k,b_k$ be $2k$ distinct random indices taken from $[n]$, and define
\[ D = \{a_1,b_1\} \times \cdots \times \{a_k,b_k\} \subseteq \binom{[n]}{k}. \]
 Clearly
\[ \EE_D \|f|_D - \fun|_D\|^2 = \|f-\fun\|^2 \leq \epsilon. \]
 Using the mapping $\{a_1,b_1\} \times \cdots \times \{a_k,b_k\} \approx \{0,1\} \times \cdots \times \{0,1\} = \{0,1\}^k$, we can think of $D$ as a $k$-dimensional Boolean cube. Under this encoding, $f|_D$ is a Boolean function $\{0,1\}^k \to \{0,1\}$, and
\[ \fun|_D(y_1,\ldots,y_k) = \sum_{i=1}^n c_{a_i} + \sum_{i=1}^n (c_{b_i}-c_{a_i}) y_i. \]
 Since $y_i = (1-(-1)^{y_i})/2$, we see that $2\widehat{\fun|_D}(\{i\}) = c_{a_i}-c_{b_i}$. Lemma~\ref{lem:fkn-consequence} therefore shows that
\[ \sum_{i=1}^k \dist(c_{b_i} - c_{a_i},\{0,\pm 1\})^2 = O(\|f|_D-\fun|_D\|^2). \]
 The lemma now follows by taking the expectation over the choice of $D$.
\end{proof}

\subsubsection{A variance formula}

An estimate of the type given by Lemma~\ref{lem:random-sub-cube} is useful since it can potentially bound the variance of $\fun$, as the following lemma shows.

\begin{lemma} \label{lem:fun-variance}
 For $\gun = \sum_i d_i x_i$ we have
\[
 \VV \gun = \frac{k(n-k)}{2(n-2)} \EE_{i \neq j} (d_i - d_j)^2 \leq k \EE_{i \neq j} (d_i - d_j)^2.
\]
\end{lemma}
\begin{proof}
 By shifting all coefficients $d_i$, we can assume without loss of generality that $\EE \gun = 0$ and so $\sum_{i=1}^n d_i = 0$ (this does not affect the quantities $d_i - d_j$). For every $i \neq j$ we have $\EE x_i = \EE x_i^2 = p$ and $\EE x_i x_j = \frac{k-1}{n-1} p$. Therefore
\[
 \VV \gun = \EE \gun^2 = p\sum_i d_i^2 + \frac{k-1}{n-1}p \sum_i d_i \sum_{j \neq i} d_j = p\left(1-\frac{k-1}{n-1}\right) \sum_i d_i^2 = \frac{k(n-k)}{n(n-1)} \sum_i d_i^2.
\]
 On the other hand,
\[
 \EE_{i \neq j} (d_i-d_j)^2 = \frac{2}{n} \sum_i d_i^2 - \frac{2}{n(n-1)} \sum_i d_i \sum_{j \neq i} d_j = \left(1-\frac{1}{n-1}\right)\frac{2}{n} \sum_i d_i^2 =
 \frac{2(n-2)}{n(n-1)} \sum_i d_i^2.
\]
We conclude that
\[
 \VV \gun = \frac{k(n-k)}{n(n-1)} \sum_i d_i^2 = \frac{k(n-k)}{2(n-2)} \EE_{i \neq j} (d_i-d_j)^2. \qedhere
\]
\end{proof}

As a corollary, we obtain the following criterion for approximating $f$ by an affine function.

\begin{corollary} \label{cor:approx}
 Suppose $d_1,\ldots,d_n$ are coefficients satisfying $k \EE_{i \neq j} (d_i-d_j)^2 = O(\epsilon)$, and define $g = \sum_i (c_i - d_i) x_i$. Then for some $m$, $\|f - (g+m)\|^2 = O(\epsilon)$.
\end{corollary}
\begin{proof}
 Define $\gun = \fun - g = \sum_i d_i x_i$. Lemma~\ref{lem:fun-variance} shows that $\VV \gun = O(\epsilon)$, and so
\[ \|f - (g+\EE \gun)\|^2 \leq 2\|f - \fun\|^2 + 2\|\gun - \EE \gun\|^2 = O(\epsilon). \qedhere \]
\end{proof}

\subsection{The case \texorpdfstring{$\epsilon < p/128$}{epsilon < p/128}}

As an application of the corollary, we show that when $\epsilon < p/128$, the function $f$ can be approximated by a function of the form $\pm \sum_{i \in S} x_i + m$. The condition $\epsilon < p/128$ ensures that the estimate of Lemma~\ref{lem:cf-soft} is strong enough to deduce $|d_i - d_j| = \dist(c_i - c_j,\{0,\pm 1\})$ for appropriate $d_i$ chosen according to the lemma.

\begin{lemma} \label{lem:sum-approx}
 If $\epsilon < p/128$ then there exist $\delta \in \{\pm 1\}$, real $m$, and a subset $S \subseteq [n]$ of size at most $n/2$, such that $\|f - (\delta \sum_{i \in S} x_i + m)\|^2 = O(\epsilon)$.
\end{lemma}
\begin{proof}
 Lemma~\ref{lem:cf-soft} shows that for some $c,d$ satisfying $|c-d|=1$ there exists a subset $S \subseteq [n]$ of size at most $n/2$ such that for $i \in S$, $|c_i-c|^2 \leq 8\epsilon/p$, and for $i \notin S$, $|c_i-d|^2 < 8\epsilon/p$. Let $\delta = d - c \in \{\pm 1\}$, and define $\gun = \fun + \delta \sum_{i \in S} x_i$. Note that $\gun = \sum_i d_i x_i$, where $d_i = c_i + \delta$ for $i \in S$ and $d_i = c_i$ for $i \notin S$. In both cases, $|d_i-d| \leq \sqrt{8\epsilon/p} < 1/4$, and so $|d_i-d_j| < 1/2$ for all $i,j$. Since $d_i-d_j = c_i - c_j + \kappa$ for some $\kappa \in \{0,\pm 1\}$, this shows that $\dist(c_i-c_j,\{0,\pm 1\}) = |d_i-d_j|$, and so Lemma~\ref{lem:random-sub-cube} implies that
\[
 k \EE_{i \neq j} (d_i-d_j)^2 = O(\epsilon),
\]
 The lemma now follows from Corollary~\ref{cor:approx}.
\end{proof}

The next step is to determine when a function of the form $\pm \sum_{i \in S} x_i + m$ can be close to Boolean. The idea is to analyze the hypergeometric random variable $\sum_{i \in S} x_i$ using the following lemma, whose somewhat technical proof appears in Section~\ref{sec:hypergeometric}.

\begin{lemma} \label{lem:hypergeometric-0}
 There exists a constant $\gamma_0 > 0$ such that the following holds, for all $k \leq n/2$. Consider the random variable $X = \sum_{i \in [t]} x_i$. If $t \leq n/2$ and $\Pr[X \in \{m,m+1\}] \geq 1-\gamma_0$ for some $m$ then $\Pr[X = 0] = \Omega(1)$ and $t \leq (3/2)p^{-1}$.
\end{lemma}

The lemma (and a few of its consequences) doesn't use the condition $k \geq 2$, and we will need this fact in the following section.

We can now determine which functions of the form $\pm \sum_{i \in S} x_i + m$ are close to Boolean.
It is enough to consider the case $\sum_{i \in S} x_i + m$, the other case following by considering the function $1-f$ instead.

\begin{lemma} \label{lem:hypergeometric-1}
 There exists a constant $\gamma_1 > 0$ such that the following holds, for all $k \leq n/2$. If $\gamma \triangleq \|f - (\sum_{i \in S} x_i + m)\|^2 \leq \gamma_1$, where $|S|\leq n/2$, then for some $\mu \in \{0,1\}$, $\|f - (\sum_{i \in S} x_i+\mu)\|^2 = O(\gamma)$. Furthermore, $|S| \leq (3/2)p^{-1}$ and $|m-\mu| = O(\sqrt{\gamma})$.
\end{lemma}
\begin{proof}
 Let $X = \sum_{i \in S} x_i$, and define $\mu$ to be the integer closest to $m$. Since $\Pr[X \notin \{-\mu, 1-\mu\}] \leq 4\|f - (\sum_{i \in S} x_i + m)\|^2 \leq 4\gamma_1$, Lemma~\ref{lem:hypergeometric-0} shows (assuming $4\gamma_1 \leq \gamma_0$) that $\Pr[X = 0] = \Omega(1)$ and $|S| \leq (3/2)p^{-1}$. This implies that $\mu \in \{0,1\}$, since otherwise $\gamma = \Omega(\Pr[X = 0]) = \Omega(1)$. Furthermore, $\gamma = \Omega(|m-\mu|^2)$ and so $|m-\mu| = O(\sqrt{\gamma})$. For any non-zero integer $z$, we have $|z-\mu| \leq |z-m| + |m-\mu| \leq 2|z-m|$, since $\mu$ is the integer closest to $m$. This shows that $\|f - (\sum_{i \in S} x_i+\mu)\|^2 \leq 4\gamma$.
\end{proof}

\begin{corollary} \label{cor:hypergeometric-1}
Let $\gamma_1$ be the constant in Lemma~\ref{lem:hypergeometric-1}. For all $k \leq n/2$, the following holds.
If $\gamma \triangleq \|f - (\sum_{i \in S} x_i + m)\|^2 \leq \gamma_1$, where $|S| \geq n/2$, then for some $\mu \in \{-k,1-k\}$, $\|f - (\sum_{i \in S} x_i + \mu)\|^2 = O(\gamma)$. Furthermore, $|S| \geq n - (3/2)p^{-1}$ and $|m-\mu| = O(\sqrt{\gamma})$.
\end{corollary}
\begin{proof}
%If $|S| \leq n/2$ then this is just Lemma~\ref{lem:hypergeometric-1}.
Suppose that $|S| > n/2$.
Since $\sum_{i \in S} x_i = k - \sum_{i \notin S} x_i$ we have
\[
(1 - f) - (\sum_{i \notin S} x_i + 1 - k - m) =
-[f - (\sum_{i \in S} x_i + m)].
\]
Therefore Lemma~\ref{lem:hypergeometric-1}, applied to $f' \triangleq 1 - f$ (a Boolean function), $S' \triangleq \overline{S}$, and $m' \triangleq 1 - k - m$, shows that for some $\mu' \in \{0,1\}$,
\[
 \|(1 - f) - (\sum_{i \notin S} x_i + \mu')\|^2 = O(\gamma).
\]
Taking $\mu = 1 - k - \mu'$, we deduce that $\|f - (\sum_{i \in S} x_i + \mu)\|^2 = O(\gamma)$.
Furthermore, $n - |S| \leq (3/2)p^{-1}$ and $|m - \mu| = |m' - \mu'| = O(\sqrt{\gamma})$.
\end{proof}

Putting Lemma~\ref{lem:sum-approx} and Lemma~\ref{lem:hypergeometric-1} together, we get that $f$ or $1-f$ is $O(\epsilon)$-approximated by a function of the form $\max_{i \in S} x_i$, where $|S| = O(p^{-1})$. (When $\mu=1$, $f$ is close to a constant.)
In order to improve the bound on $|S|$, we estimate the probability that $\sum_{i \in S} x_i \geq 2$.

\begin{lemma} \label{lem:hypergeometric-2}
 Let $S$ be a subset of $[n]$ of size $t \triangleq |S| \leq (3/2)p^{-1}$. If $t \geq 2$ then
\[ \Pr\left[\sum_{i \in S} x_i \geq 2\right] = \Omega((pt)^2). \]
\end{lemma}
\begin{proof}
 Let $p' = \frac{k-1}{n-1} \geq p/2$ (using $k \geq 2$) and $p'' = \frac{k-2}{n-2} \leq p$.
 The inclusion-exclusion principle shows that
\begin{align*}
 \Pr\left[\sum_{i \in S} x_i \geq 2\right] &\geq \binom{t}{2} \Pr[x_1=x_2=1] - \binom{t}{3} \Pr[x_1=x_2=x_3=1] \\ &=
 \binom{t}{2} pp' \left(1 - \frac{tp''}{3}\right) \\ &\geq \frac{t^2}{2} \frac{p^2}{2} \frac{1}{2} = \frac{(tp)^2}{8}. \qedhere
\end{align*}
\end{proof}

We can now prove Theorem~\ref{thm:fkn-slice} when $\epsilon < p/128$.

\begin{lemma} \label{lem:fkn-slice-large-error}
 Suppose $f\colon \binom{[n]}{k} \to \{0,1\}$ is $\epsilon$-close to an affine function, and let $p = \min(k/n,1-k/n)$. If $2 \leq k \leq n-2$ and $\epsilon < p/128$ then either $f$ or $1-f$ is $O(\epsilon)$-close to $\max_{i \in S} x_i$ for some set $S$ of size at most $\max(1,\sqrt{\epsilon}/p)$.
\end{lemma}
\begin{proof}
 Lemma~\ref{lem:sum-approx} shows that for some real $m$ and set $S$ of size at most $n/2$ we have $\|f - (\delta \sum_{i \in S} x_i + m)\|^2 = O(\epsilon)$. For simplicity, assume that $\delta = 1$ (when $\delta = -1$, consider $1-f$ instead of $f$). Lemma~\ref{lem:hypergeometric-1} then implies that $\|f - \left(\sum_{i \in S} x_i + \mu\right)\|^2 = O(\epsilon)$ for some $\mu \in \{0,1\}$, assuming $\epsilon$ is small enough (otherwise the lemma is trivially true, since every Boolean function is $1$-close to the constant $0$ function), and moreover $|S| \leq (3/2)p^{-1}$.

 Suppose first that $\mu = 0$. In this case, if $|S| \geq 2$ then Lemma~\ref{lem:hypergeometric-2} implies that $(p|S|)^2 = O(\epsilon)$ and so $|S| = O(\sqrt{\epsilon}/p)$. The function $g = \max_{i \in S} x_i$ results from rounding $h \triangleq \sum_{i \in S} x_i$ to Boolean, and so $\|f - g\|^2 = O(\|f - h\|^2) = O(\epsilon)$. When $\mu = 1$, we similarly get $\|f - 1\|^2 = O(\epsilon)$.
\end{proof}

\subsection{The case \texorpdfstring{$\epsilon = \Omega(p)$}{epsilon = Omega(p)}}

We move on to the case $\epsilon = \Omega(p)$.
In this case the analog of Lemma~\ref{lem:sum-approx} states that $f$ can be approximated by a function of the form $\sum_{i \in S_+} x_i - \sum_{i \in S_-} x_i + m$, where at least one of $S_+,S_-$ is small.

\begin{lemma} \label{lem:sum-diff-approx}
 There exist real $m$ and two subsets $S_+,S_- \subseteq [n]$ satisfying $\frac{|S_+|}{n} \frac{|S_-|}{n} = O(\epsilon/k)$ such that $\|f - (\sum_{i \in S_+} x_i - \sum_{i \in S_-} x_i + m)\|^2 = O(\epsilon)$.
\end{lemma}
\begin{proof}
 Lemma~\ref{lem:random-sub-cube} shows that $k\EE_{j \neq i} \dist(c_i-c_j,\{0,\pm 1\})^2 = O(\epsilon)$. This implies that for some $i_0 \in [n]$,
\[
 k \EE_{j \neq i_0} \dist(c_{i_0}-c_j,\{0,\pm 1\})^2 = O(\epsilon).
\]
 We partition the coordinates in $[n]$ into four sets. For $\delta \in \{0,\pm 1\}$, we let $S_\delta = \{j \in [n] : |c_j-c_{i_0}-\delta| < 1/4\}$, and we put the rest of the coordinates in a set $R$. Since $\EE_{j \neq i_0} \dist(c_{i_0}-c_j,\{0,\pm 1\})^2 = \Omega(\frac{|R|}{n})$, we conclude that $\frac{|R|}{n} = O(\epsilon/k)$. Since $\EE_{i \neq j} (c_i-c_j)^2 = \Omega(\frac{|S_{-1}|}{n} \frac{|S_{+1}|}{n})$, we conclude that $\frac{|S_{-1}|}{n} \frac{|S_{+1}|}{n} = O(\epsilon/k)$.

 Define now $d_i = c_i - 1$ for $i \in S_{+1}$, $d_i = c_i + 1$ for $i \in S_{-1}$, and $d_i = c_i$ otherwise. When $i,j \in S_0 \cup S_{+1}$ or $i,j \in S_0 \cup S_{-1}$, we get $|d_i - d_j| < 1/2$, and since $d_i - d_j = c_i - c_j + \kappa$ for some $\kappa \in \{0,\pm 1\}$, we conclude that $\dist(c_i-c_j,\{0,\pm 1\}) = |d_i - d_j|$. For all $i,j$ we claim that $(d_i-d_j)^2 \leq 7\dist(c_i-c_j,\{0,\pm 1\})^2 + 16$. Indeed, for some $\kappa \in \{0,\pm 1,\pm 2\}$ we have $d_i - d_j = c_i - c_j + \kappa$. If $|c_i-c_j| \leq 2$ then $(d_i-d_j)^2 \leq 16$, whereas if $|c_i-c_j| \geq 2$, say $c_i-c_j \geq 2$, then $c_i-c_j-1 \leq (c_i-c_j-1)^2$ and so
\begin{align*}
 (d_i-d_j)^2 &= ((c_i-c_j-1)+(\kappa+1))^2 \\ &= \dist(c_i-c_j,\{0,\pm 1\})^2 + 2(c_i-c_j-1)(\kappa+1) + (\kappa+1)^2 \\ &\leq 7\dist(c_i-c_j,\{0,\pm 1\})^2 + 9.
\end{align*}

 Call a pair of indices $i,j$ \emph{good} if $i,j \in S_0 \cup S_{+1}$ or $i,j \in S_0 \cup S_{-1}$, and note that the probability that $i,j$ is bad (not good) is at most $2\frac{|R|}{n} + 2\frac{|S_{+1}|}{n}\frac{|S_{-1}|}{n} = O(\epsilon/k)$.
 This shows that
\[
 k \EE_{i \neq j} (d_i-d_j)^2 \leq 7k\EE_{i \neq j} \dist(c_i-c_j,\{0,\pm 1\})^2 + 16k\Pr[i,j\text{ bad}] = O(\epsilon).
\]
 Corollary~\ref{cor:approx} now completes the proof.
\end{proof}

An argument similar to the one in Lemma~\ref{lem:fkn-slice-large-error} completes the proof of the theorem.

\begin{proof}[Proof of Theorem~\ref{thm:fkn-slice}]
 If $\epsilon \leq p/128$ then the result follows from Lemma~\ref{lem:fkn-slice-large-error}, so we can assume that $\epsilon > p/128$, and in particular we can assume that $p \leq 1/5$, since otherwise $\epsilon > 1/640$ and so the result is trivial (since every Boolean function is $1$-close to the constant~$0$ function).
In several other places in the proof we also assume that $\epsilon$ is small enough (smaller than some universal constant independent of $p$); otherwise the result is trivial.
 % Is this needed?
 %We can also assume that $n$ is large enough. Indeed, for any fixed $n$, since there are only finitely many values of $k$ and finitely many functions on at most $n$ coordinates, considering the best approximation $\fun$, either $\epsilon = 0$, in which case the result follows from Lemma~\ref{lem:cf-hard}, or $\epsilon = \Omega(1)$, in which case the result is trivial.

Lemma~\ref{lem:sum-diff-approx} shows that for some real $m$ and sets $S_+,S_-$ it holds that $\|f - (\sum_{i \in S_+} x_i - \sum_{i \in S_-} x_i + m)\|^2 = O(\epsilon)$. By possibly replacing $f$ by $1-f$, we can assume that $|S_-| \leq |S_+|$. In particular, $|S_-| \leq n/2$, and so $k/(n-|S_-|) \leq 2p < 1/2$. Note that $S_-$ could be empty.

We now consider two different cases: $|S_+| \leq (n-|S_-|)/2$ and $|S_+| \geq (n-|S_-|)/2$.

\paragraph{Case 1: $|S_+| \leq (n-|S_-|)/2$.}
Consider some setting of the variables in $S_-$ which sets $w$ of them to $1$. This setting reduces the original slice to a slice $Z$ isomorphic to $\binom{[n']}{k'}$, where $n' = n-|S_-|$ and $k' = k-w$. The corresponding $p' = \frac{k'}{n'}$ satisfies $p' \leq \frac{k}{n-|S_-|} < 1/2$.
Lemma~\ref{lem:hypergeometric-1} applied with $m \triangleq m - \sum_{i \in S_-} x_i$ shows that for each such setting, either $\|f-(\sum_{i \in S_+} x_i - \sum_{i \in S_-} x_i + m)\|_Z^2 = \Omega(1)$ or $\dist(m - \sum_{i \in S_-} x_i,\{0,1\}) \leq 1/4$ (note that the lemma works even if $k' < 2$). In the latter case, we say that $w$ is good.

If no $w$ is good then $\epsilon = \Omega(1)$, so by assuming that $\epsilon$ is small enough we can guarantee that some $w$ is good. On the other hand, the condition on $m$ shows that at most two values $w_0,w_0+1$ are good. This implies that $\Pr[\sum_{i \in S_-} x_i \in \{w_0,w_0+1\}] \geq 1 - O(\epsilon)$, and so for small enough $\epsilon$, Lemma~\ref{lem:hypergeometric-0} shows that $\Pr[\sum_{i \in S_-} x_i = 0] = \Omega(1)$. Therefore we can assume that $w_0 = 0$.

Let $\|\cdot\|_{w=W}$ denote the norm restricted to inputs in which $\sum_{i \in S_-} x_i = W$ (we similarly use $\Pr_{w=w}[\cdot]$). Since $\Pr[\sum_{i \in S_-} x_i = 0] = \Omega(1)$, we must have $\|f - (\sum_{i \in S_+} x_i + m)\|_{w=0}^2 = O(\epsilon)$. Lemma~\ref{lem:hypergeometric-1} shows that for some $\mu \in \{0,1\}$, also $\|f - (\sum_{i \in S_+} x_i + \mu)\|_{w=0}^2 = O(\epsilon)$, and moreover $|m - \mu| \leq 1/4$ and $|S_+| \leq (3/2)p'^{-1}$, where $p' = k/n' \geq p$. Lemma~\ref{lem:hypergeometric-2} implies that in fact $(p'|S_+|)^2 = O(\epsilon)$, and so $|S_+| = O(\sqrt{\epsilon}/p') = O(\sqrt{\epsilon}/p)$.

We now consider two subcases: $\mu = 0$ and $\mu = 1$.

\subparagraph{Case 1(a): $\mu = 0$.}
Since $g \triangleq \max_{i \in S_+} x_i$ is the result of rounding $\sum_{i \in S_+} x_i + \mu$ to Boolean, we conclude that $\|f - g\|_{w=0}^2 = O(\epsilon)$.
Since $|m| \leq 1/4$, it cannot be the case that $\dist(m - 1, \{0,1\}) \leq 1/4$, and so $1$ is not good. In other words, $\Pr[\sum_{i \in S_-} x_i = 0] \geq 1 - O(\epsilon)$.
This implies that $\|f - g\|^2 = O(\epsilon)$, completing the proof in this case.

\subparagraph{Case 1(b): $\mu = 1$.}
In this case $\Pr_{w=0}[\sum_{i \in S_+} x_i = 0] \geq 1 - O(\epsilon)$, since $|\sum_{i \in S_+} x_i + \mu - f(x)| \geq \sum_{i \in S_+} x_i$. This implies that $\Pr[\sum_{i \in S_+} x_i = 0] \geq 1 - O(\epsilon)$, since the assumption $\sum_{i \in S_-} x_i = 0$ only makes it harder for $\sum_{i \in S_+} x_i$ to vanish. Since $|S_+| \geq |S_-|$, this implies that $\Pr[\sum_{i \in S_-} x_i = 0] \geq 1 - O(\epsilon)$, for similar reasons. We conclude that $h \triangleq \sum_{i \in S_+} x_i - \sum_{i \in S_-} x_i$ vanishes with probability $1 - O(\epsilon)$. Since $\|f - (h + m)\|^2 = O(\epsilon)$, it follows that $\Pr[f = \nu] = 1 - O(\epsilon)$, where $\nu$ is the rounding of $m$ to $\{0,1\}$. Thus $\|f-\nu\|^2 = O(\epsilon)$, completing the proof in this case.

\paragraph{Case 2: $|S_+| \geq (n-|S_-|)/2$.} In the setup of the first case, applying Corollary~\ref{cor:hypergeometric-1} instead of Lemma~\ref{lem:hypergeometric-1} allows us to conclude that either $\|f-(\sum_{i \in S_+} x_i - \sum_{i \in S_-} x_i + m)\|_Z^2 = \Omega(1)$ or $\dist(m - \sum_{i \in S_-} x_i + k',\{0,1\}) \leq 1/4$. Since $k' = k - \sum_{i \in S_-} x_i$, we can rewrite the latter condition as $\dist(m - 2\sum_{i \in S_-} x_i + k, \{0,1\}) \leq 1/4$.
As before, in the latter case we say that $w = \sum_{i \in S_-} x_i$ is good.

In contrast to the situation in Case~1, here at most one (and so exactly one) $w$ can be good. Lemma~\ref{lem:hypergeometric-0} implies that $w = 0$, and so $\Pr[\sum_{i \in S_-} x_i = 0] \geq 1 - O(\epsilon)$. This implies that $\|f - (\sum_{i \in |S_+|} x_i + m)\|_{w=0}^2 = O(\epsilon)$. Corollary~\ref{cor:hypergeometric-1} implies that for some $\mu \in \{0,1\}$ we have $\|f - (\sum_{i \in S_+} x_i - k + \mu)\|_{w=0}^2 = O(\epsilon)$, and moreover $|S_+| \geq n' - (3/2)p'^{-1}$.

Let $T = \overline{S_+ \cup S_-}$, so that $|T| \leq (3/2)p'^{-1}$. The foregoing shows that $\|f - (-\sum_{i \in T} x_i + \mu)\|_{w=0}^2 = O(\epsilon)$ and so $\|(1-f) - (\sum_{i \in T} x_i + 1-\mu)\|_{w=0}^2 = O(\epsilon)$. As in Case~1, Lemma~\ref{lem:hypergeometric-2} implies that in fact $|T| = O(\sqrt{\epsilon}/p)$.

If $\mu = 0$ then let $g \triangleq 1$, and if $\mu = 1$ let $g \triangleq \max_{i \in T} x_i$. In both cases $g$ is the Boolean rounding of $\sum_{i \in T} x_i + 1-\mu$, and so $\|(1-f) - g\|_{w=0}^2 = O(\epsilon)$. It follows that $\|(1-f) - g\|^2 = O(\epsilon)$, completing the proof.
\end{proof}

\subsection{Hypergeometric estimate} \label{sec:hypergeometric}

To complete the proof of Theorem~\ref{thm:fkn-slice}, we present the rather technical proof of Lemma~\ref{lem:hypergeometric-0}.

\begin{proof}[Proof of Lemma~\ref{lem:hypergeometric-0}]
  The distribution of $X$ is given by
\[
 \Pr[X = s] = \frac{\binom{t}{s}\binom{n-t}{k-s}}{\binom{n}{k}}.
\]
 If $t \leq 3$ then $\Pr[X = 0] = \Omega_{t}((1-p)^{t}) = \Omega_{t}(1)$ and $t \leq (3/2)2 \leq (3/2)p^{-1}$. Similarly, if $k \leq 3$ then $\Pr[X = 0] = \Omega_k((1-t/n)^k) = \Omega_k(1)$ and $t \leq n/2 \leq 3 \leq (3/2)p^{-1}$. We can therefore assume that $t,k \geq 4$.

 In view of showing that the mode of $X$, given by the classical formula $s_0 = \lfloor \frac{(k+1)(t+1)}{n+2} \rfloor$, is attained at zero, assume that $s_0 \geq 1$. Note that $s_0+2 \leq k$ since otherwise $s_0 \geq k-1$ and so $(k+1)(t+1) \geq (k-1)(n+2)$, implying $t+1 \geq (3/5)(n+2) \geq (6/5)(t + 1)$, which is impossible. Similarly, $s_0+2 \leq t$.

 A simple calculation shows that $\rho_s \triangleq \frac{\Pr[X = s+1]}{\Pr[X = s]} = \frac{(t-s)(k-s)}{(s+1)(n-t-k+s+1)}$. Therefore
\begin{align*}
 \frac{\rho_{s+1}}{\rho_s} &= \frac{t-s-1}{t-s} \frac{k-s-1}{k-s} \frac{s+1}{s+2} \frac{n-t-k+s+1}{n-t-k+s+2} \\ &=
 \left(1-\frac{1}{t-s}\right) \left(1-\frac{1}{k-s}\right) \left(1-\frac{1}{s+2}\right) \left(1-\frac{1}{n-t-k+s+2}\right).
\end{align*}
 Since $t-s_0 \geq 2$, $k-s_0 \geq 2$, $s_0+2 \geq 2$ and $n-t-k+s_0+2 \geq 2$ (using $t,k \leq n/2$), we conclude that $\rho_{s_0+1}/\rho_{s_0} \geq 1/16$. This shows that $\Pr[X=s_0+2] = \rho_{s_0+1} \rho_{s_0} \Pr[X=s_0] \geq \frac{\rho_{s_0}^2}{16} \Pr[X=s_0]$.

 If $\Pr[X=s_0] \leq 1/3$ then $\Pr[X \in \{m,m + 1\}] \leq 2/3$ and so $\gamma \geq 1/3$. We can therefore assume that $\Pr[X=s_0] \geq 1/3$, and so $\Pr[X=s_0+2] \geq (\rho_{s_0}^2/16)(1/3) = \Omega(\rho_{s_0}^2)$. Since $\{m,m+1\}$ cannot contain both $s_0$ and $s_0+2$, this shows that $\gamma_0 = \Omega(\rho_{s_0}^2)$, implying that we can assume that $\rho_{s_0} < \tau_0$ for some small $\tau_0$.

 Suppose now that $s_0 \geq 1$. Then $s_0 > \frac{(k+1)(t+1)}{n+2} - 1$, and so
\begin{align*}
 \rho_{s_0} &\geq \frac{(t-\frac{(k+1)(t+1)}{n+2})(k-\frac{(k+1)(t+1)}{n+2})}{\frac{(k+1)(t+1)}{n+2}(n-t-k+\frac{(k+1)(t+1)}{n+2})} \\ &\geq \Omega(1) \frac{(t+1)(1-\frac{k+1}{n+2})(k+1)(1-\frac{t+1}{n+2})}{\frac{(k+1)(t+1)}{n+2} \cdot O(n)} = \Omega(1).
\end{align*}
 By choosing $\tau_0$ (and so $\gamma_0$) appropriately, we can conclude that $s_0 = 0$, which shows that $1 > \frac{(k+1)(t+1)}{n+2} > pt$, and so $t < p^{-1}$.
 Moreover, $\Pr[X=0] = \Pr[X=s_0] \geq 1/3$.
\end{proof}

\bibliographystyle{plain}
\bibliography{FKN-Wilson}

\end{document}